\documentclass{amsart}
\usepackage{amssymb, amsmath, amsthm, verbatim}
\newtheorem{theorem}{Theorem}[section]
\newtheorem{lemma}[theorem]{Lemma}
\newtheorem{proposition}[theorem]{Proposition}

\newtheorem{corollary}[theorem]{Corollary}

\theoremstyle{definition}

\theoremstyle{remark}

\numberwithin{equation}{section}

\newcommand{\dash}{\text{-}}
\newcommand{\w}{\mathrm{w}}
\newcommand{\wst}{\mathrm{w}^*}

\newcommand{\dens}{\mathrm{dens}}
\renewcommand{\span}{\mathrm{span}}
 
\newcommand{\dist}{\mathrm{dist}}
\newcommand{\conv}{\mathrm{conv}}

\newcommand{\R}{\mathbb{R}}

\newcommand{\X}{\mathrm{X}}

\newcommand{\Y}{\mathrm{Y}}
\newcommand{\Z}{\mathrm{Z}}
\newcommand{\B}{\mathbf{B}}
\newcommand{\I}{\mathbf{I}}

\renewcommand{\S}{\mathbf{S}}

\renewcommand{\ker}{\mathrm{Ker}}
\renewcommand{\mod}{/}

\newcommand{\cf}{\mathrm{cf}}

\newcommand{\codens}{\mathrm{codens}}
\newcommand{\SD}{\mathrm{SD}}
\begin{document}

\title[Unconditional and bimonotone structures]{Unconditional and bimonotone structures in 
high density Banach spaces}

\begin{abstract}
It is shown that every normalized weakly null sequence of length $\kappa_{\lambda}$ in a Banach space has a subsequence of length $\lambda$ which is an unconditional basic sequence;
here $\kappa_{\lambda}$ is a large cardinal depending on a given infinite cardinal $\lambda$. 
Transfinite topological games on Banach spaces are analyzed which determine the existence of a long unconditional basic sequence. Then 'asymptotic disentanglement' condition in a transfinite setting is studied which ensures a winning strategy for the unconditional basic sequence builder in the above game. The following problem is investigated: When does a Markushevich basic sequence with length uncountable regular cardinal $\kappa$ admit a subsequence of the same length which is a bimonotone basic sequence? Stabilizations of projectional resolutions of the identity (PRI) are performed under a density contravariance principle in order  to gain some additional strong regularity properties, such as bimonotonicity.
\end{abstract}

\author{Jarno Talponen}
\address{University of Eastern Finland, Mathematics, Box 111, FI-80101 Joensuu, Finland} 
\email{talponen@iki.fi}
\date{\today}
\maketitle

\section{Introduction}


Schauder bases are very useful both in the theoretical study of Banach spaces, as well as in the applications of Functional Analysis. Many Banach spaces exhibit a rich system of projections (see e.g. \cite{Kubis}) while some spaces admit few operators (e.g. \cite{Ko}). These considerations can be applied in classifying spaces, such as hereditarily indecomposable spaces, through the spectacle of Gowers' program, see e.g. \cite{AT2, Gowers, Ko2, pelczar, FR}. In practice, it is often easiest to work within separable complemented fragments of a given Banach space. Even if a space fails to have a Schauder basis it is often possible and convenient to work with some weaker form of coordinate system, such as a Markushevich basis or a Projectional Resolution of the Identity. Also, a space failing to admit a nice Schauder basis may still contain a sizeable subspace having one. 

Similarly, if one is given a suitably dispersed (possibly transfinite) sequence of vectors it is a very natural idea to attempt to extract a subsequence which is a Schauder basic sequence of some sort. Based on 'common sense' combinatorial considerations, the longer the transfinite sequence to begin with,
the easier it is to find a countable unconditional subsequence. This elegant problem has been studied extensively, often involving spaces of high density and infinitary combinatorial assumptions which are extraneous to the standard axioms of the set theory, see e.g. \cite{DLT, ketonen, LoTo2}. More generally, ideas originating from various branches of mathematical logic have been very fruitful in the study of the geometry of Banach spaces, see e.g.
\cite{AD, AG, AT, AT2, To}.

There is an example of a reflexive Banach space $\X$ of density $\aleph_1$ with a Schauder basis but 
such that the space does not contain any (infinite) unconditional basic sequence, see 
\cite{ALT}, cf. \cite{AG,AT,AT2}. On the other hand, it is known that in a fairly general setting a space having density uncountable regular $\kappa$ has a monotone basic sequence (see e.g. \cite{Tal2}). The condition of bimonotonicity of a basic sequence is an intermediate notion between monotonicity and unconditionality, apparently closer to monotonicity. Therefore, it seems reasonable to ask if well-behaved non-separable Banach spaces admit long bimonotone structures and these are studied here at the end. 

The main motivating question in this paper is the open problem regarding the exact value of the cardinal invariants $\mathfrak{nc}$ and $\mathfrak{nc}_{\mathrm{rfl}}$, investigated e.g. in \cite{DLT, ketonen}. These are the least cardinals $\kappa$ such that all Banach spaces (resp. reflexive spaces) with density 
$\geq\! \kappa$ admit an unconditional basic sequence. In fact, it is not currently known if these are non-large cardinals absolutely. In this connection, we obtain a natural upper bound $\mathfrak{c}^+$ for a reasonable class of spaces.

We aim to extend the work in \cite{DLT} and \cite{Tal2}; the former deals e.g. with weakly null sequences of vectors, combinatorial principles and countable unconditional basic sequences, while the latter involves somewhat dispersed long sequences of vectors, topological tightness conditions and long monotone basic sequences. In this paper we use both `pure' combinatorial arguments and `geometro-combinatorial' reasoning in extracting or discovering \emph{long} basic sequences. We investigate different types of bases here, most importantly of the suppression unconditional and the bimonotone sort. 

To highlight some considerations here, firstly we apply a separable reduction with 
infinitary Ramsey theory to extract long unconditional subsequences of 
normalized weakly null sequences having the lengths of suitable large cardinals. The main point here is refining the analysis in order to weaken the combinatorial hypothesis. Apart from combinatorial analysis, we will also analyze Banach spaces whose subspaces of relatively small density are transfinitely asymptotically free; in a sense 'disentangled' from the bulk of the space. We then observe that such a principle yields a winning strategy for a builder of long unconditional bases in a transfinite $2$-player topological game played on the given space. Thus we relax the burden upon combinatorics by imposing on a Banach space itself some structural assumptions having some combinatorial flavor. We continue applying and developing the $(\sigma)$ property introduced in \cite{Tal2} in a Baire category spirit, stating that coseparable closed subspaces are preserved in countable intersections. This property holds in Weakly Lindel\"{o}f Determined spaces. To borrow some terminology from group theory, some conditions of Banach spaces hold up \emph{virtually}, or up to quotients by
negligible subspaces. The $(\sigma)$ property can be used in virtually 'amalgamating' countably many conditions that hold virtually. To obtain one of the main results in this paper, we will, in a sense, dualize this property, which then enables us to build reverse-monotone basic sequences. 

\subsection{Preliminaries}

Real infinite-dimensional Banach spaces are typically denoted here by $\X,\Y,\Z$. 
We denote by $\B_{\X}$ the closed unit ball of $\X$ and by $\S_{\X}$ the unit sphere of $\X$. 
Here $\kappa$ always stands for an \emph{uncountable regular} cardinal. Usually
$\lambda$ denotes a cardinal or a limit ordinal and usually we will apply cardinal arithmetic (instead of
ordinal arithmetic) operations. The use of $\aleph_\alpha$ notation occurs on ad hoc typographical grounds. 
We use the symbols $\bigvee$ and $\bigwedge$ for the least upper bound and the greatest lower bound 
(also $\vee$, $\wedge$ for disjunction, conjunction; or $\max$, $\min$), respectively. A subsequence $\{\alpha_{\sigma}\}_{\sigma<\beta} \subset \lambda$
is a mapping $\beta \to \lambda,\ \sigma\mapsto \alpha_\sigma$, and, unless otherwise stated, subsequences are assumed strictly increasing here.

See \cite{EMHR, HHZ, Jech, Ne, SingerI, SingerII, Z} for suitable background information on the standard concepts in infinitary combinatorics, Banach spaces, set theory, topology of Banach spaces and Schauder bases, respectively, cf. \cite{BP, Gr, Kalenda2000, Sh, Sh2, To}.  We assume a familiarity with biorthogonal systems of non-separable Banach spaces and closely related matters; see the monograph \cite{Hajek_biortsyst} for an exposition.
 
Denote $\dist(A,B)=\inf_{a\in A, b\in B}\|a-b\|$, where $A,B\subset \X$. The weak and weak-star topology
of a Banach space are denoted by $\w$ and $\wst$, respectively.
Recall that $\mathcal{F}\subset \X^{\ast}$ is a separating subset, if and only if for each $x\in\X$ there is $f\in\mathcal{F}$ such that $f(x)\neq 0$, if and only if $\overline{[\mathcal{F}]}^{\wst}=\X^{\ast}$.
We denote by $\ell_{<\lambda}^\infty (\kappa)$ the subspace of $\ell^\infty (\kappa)$ of sequences
whose support has cardinality $<\!\lambda$, always implicitly assuming that $\cf(\lambda )> \omega$ (to ensure the completeness of the subspace). Denote $\ell_{\lambda}^\infty (\kappa) = \ell_{<\lambda^+}^\infty (\kappa)$.
We will apply the facts presented below frequently and sometimes implicitly.
\subsubsection{Schauder basic sequences}
Given a limit ordinal $\lambda$, a sequence $\{x_{\alpha}\}_{\alpha<\lambda}\subset\X \setminus \{0\}$ is said to be a (transfinite) basic sequence, if there
is a constant $1\leq C<\infty$ such that $\|y\| \leq C \|y+z\|$
for all $y\in [x_{\alpha}:\ \alpha<\eta]$, $z\in [x_{\alpha}:\ \eta\leq \alpha<\lambda]$ and all $\eta<\lambda$ (see \cite[p.589]{SingerII}). In such
case there are natural linear basis projections
$P_{\eta}\colon [x_{\alpha}:\ \alpha<\lambda]\to [x_{\alpha}:\ \alpha<\eta]$ such that 
$\|P_{\eta}\|\leq C$ for $\eta<\lambda$. If $[e_\alpha \colon \alpha<\lambda] =\X$ then the sequence is the Schauder basis of $\X$. The limit ordinal $\lambda$ is called \emph{length} or \emph{order type} of the basic sequence. The basic sequence is said to be \emph{monotone}, if the basis projections are contractive, i.e. $C=1$. The basic sequence is said to be \emph{reverse monotone} if the corresponding coprojections are norm-$1$ and \emph{bimonotone} if it is both monotone and reverse monotone. Recall that a sequence $\{e_\alpha \}_{\alpha<\lambda} \subset \S_{\X}$ is  $C$-\emph{suppression unconditional basic sequence}, $C\geq 1$, if 
\[\left\|\sum_{\alpha \in A} c_\alpha e_\alpha \right\| \leq  C\left\|\sum_{\alpha \in B} c_\alpha e_\alpha \right\|\]
for any finite subsets $A \subset B \subset \lambda$ and any choices of coefficients $c_\alpha \in \R$, $\alpha\in B$. A stronger form of unconditionality, $C$-unconditional basic sequnce satisfies
\[\left\|\sum_{\alpha \in A} c_\alpha e_\alpha \right\| \leq  C\left\|\sum_{\alpha \in A} \theta(\alpha) c_\alpha e_\alpha \right\|\]
for any sequence of signs $\theta \in \{\pm 1\}^\lambda$. For a yet stronger form, we call here a sequence \emph{modular unconditional} basic sequence if 
\[\left\| \sum_{\alpha \in A} c_\alpha e_\alpha \right\| \leq \left\| \sum_{\alpha \in A} d_\alpha e_\alpha \right\|
\quad \implies\quad  \left\| \sum_{\alpha \in A} c_\alpha e_\alpha + e_\gamma \right\| \leq \left\| \sum_{\alpha \in A} d_\alpha e_\alpha + e_\gamma \right\|\]
for finite $A \subset \lambda$, $A\not\ni \gamma <\lambda$. By induction, using $\pm c_\alpha = \mp  d_\alpha$, one can check that a modular unconditional sequence is $1$-unconditional as well.

\subsubsection{Corson's property and tightness} Recall that a topological space $(T,\tau)$ 
is \emph{countably tight}, if and only if for any increasing sequence $E_{\alpha}\subset T$, $\alpha<\mu$,  $\cf(\mu)>\omega$, of closed sets, the union $\bigcup_{\alpha<\mu }E_{\alpha}\subset T$ is closed.
The Banach spaces are countably tight in their norm and weak topologies. Roughly speaking, the well-behaved dual spaces are countably tight in their weak-star topology. We frequently apply Corson's 
property (C) which is a convex generalization of weak-star countable tightness (see \cite{Corson, Pol}):
A Banach space $\X$ has property (C) if for any $\Gamma \subset \X^*$ and $f \in \overline{\Gamma}^{\wst}$ there is already a subset $\Gamma_0 \subset \Gamma$, $|\Gamma_0 | \leq \aleph_0$, 
with $f\in \overline{\conv}^{\wst} (\Gamma_0 )$.

\subsubsection{Markushevich-bases and dispersed sequences}
Suppose that $\{(x_{\alpha},x^{\ast}_{\alpha})\}_{\alpha<\lambda}\subset \X\times \X^{\ast}$ 
is a biorthogonal system, i.e. $x_{\alpha}^{\ast}(x_{\beta})=\delta_{\alpha,\beta}$. We call $\{x_{\alpha}\}_{\alpha<\lambda}$ a \emph{biorthogonal
sequence} for brevity. If $[x_{\alpha}: \alpha<\lambda]=\X$ and $\overline{[x^{\ast}_{\alpha}:\ \alpha<\lambda]}^{\wst}=\X^{\ast}$, then
$\{(x_{\alpha},x^{\ast}_{\alpha})\}_{\alpha<\lambda}$ is called a \emph{Markushevich basis} or \emph{M-basis}.
If $\{(x_{\alpha},x^{\ast}_{\alpha})\}_{\alpha<\lambda}$ is an M-basis on $[x_{\alpha}: \alpha<\lambda]$, then $\{x_{\alpha}\}_{\alpha<\lambda} \subset \X$ is called
an M-basic sequence.
We often do not explicitly include the functionals in M-basis. Coined and investigated in \cite{Tal2}, a sequence $\{x_\alpha\}_{\alpha<\lambda}\subset \X$ is called 
\emph{strongly dispersed} ($\SD$) if 
\[\bigcap_{\nu<\lambda} [x_\alpha \colon \nu<\alpha<\lambda] =\{0\}.\]
This can be seen as an opposite condition to overfilling sequences. For example, if 
$\{x_\alpha\}_{\alpha<\kappa}\subset \X$ is an M-basic sequence or weakly convergent to $0$, 
then it is $\SD$. There was an unfortunate choice of 
terminology in \cite{Tal2} where the author used `weakly null' and `weakly convergent to $0$' as synonyms, 
meaning the latter. Recall that $x_\alpha \stackrel{\w}{\longrightarrow} 0$, as $\alpha \to \lambda$ if 
$f(x_\alpha ) \to 0$ as $\alpha \to \lambda$ for each $f\in \X^*$, whereas 
$\{x_\alpha\}_{\alpha<\lambda}\subset \X$ is weakly null if $\{f (x_\alpha)\}_{\alpha<\lambda} \in c_0 (\lambda )$ for each $f \in \X^*$ 
(thus for each $(c_\alpha )_{\alpha<\lambda} \in  c_0 (\lambda )$ and 
$\varepsilon>0$ the set $\{\alpha <\lambda \colon |c_\alpha| > \varepsilon \}$ is finite). To summarize,
\[\text{weakly\ null}\ \implies\ \text{weakly\ convergent\ to}\ 0\ \implies \SD\ \Longleftarrow\  \text{M-basic}.\] 
If $\X$ has property (C), then, for sequences of $\X$ having length $\lambda$ 
with $\cf(\lambda)>\omega$, the weak convergence to $0$ is equivalent to the $\SD$ condition, see 
\cite{Tal2}. 

\subsubsection{Well-behaved Banach spaces} 
A Banach space $\X$ is Weakly Lindel\"{o}f Determined (WLD) if and only if it admits an M-basis and has 
Corson's property (C) (see \cite{Z}).
We will frequently employ the following useful equivalent condition 
(see e.g. \cite[Thm.4.17]{Kalenda_Extracta2000}): There is an M-basis $\{(x_{\alpha},g_{\alpha})\}_{\alpha}$ of $\X$ such that
\begin{equation}\label{eq: count}
|\{\alpha :\ f(x_{\alpha})\neq 0\}|\leq \aleph_{0}\quad \mathrm{for\ any}\ f\in \X^{\ast}, 
\end{equation}
and, in fact, then \emph{any} M-basis on $\X$ has this property according to property (C). 
A subspace $\Y\subset\X^{\ast}$ is called $r$-norming, $0< r\leq 1$, if $\inf_{x\in\S_{\X}}\sup_{x^{\ast}\in \S_{\Y}}x^{\ast}(x)\geq r$, or, equivalently, if
$r\B_{\X^{\ast}}\subset \overline{\B_{\Y}}^{\wst}\subset \X^*$. We denote by
$\dens(\X)$ the density of $\X$ and $\wst\dash\dens(\X^* )$ the density of a dual space $\X^*$ in the weak-star topology. The next notions and facts are from \cite{Tal2}. For a subspace $\Y \subset \X$ we write $\codens(\Y)=\dens(\X\mod\Y)$, the superspace $\X$ being understood. If $\codens(\Y)=\aleph_0$
then $\Y$ is \emph{coseparable} in $\X$. In a WLD space $\X$ we have 
$\codens(\bigcap_{\alpha<\lambda} \Z_\alpha ) \leq |\lambda| \vee \bigvee_{\alpha<\lambda}
\codens(\Z_\alpha )$ 
for $\Z_\alpha \subset \X$, $\alpha<\lambda$. In particular, a WLD space $\X$ has 
property $(\sigma)$, involving the lattice of its subspaces: \emph{The coseparable subspaces are preserved in countable intersections.}

\subsubsection{Projectional structure}
Recall that a Banach space $\X$ has the \emph{Separable Complementation Property} (SCP) (resp. $1$-SCP), if each separable subspace is contained in a complemented (resp. $1$-complemented) subspace of $\X$.

An M-basis $\{(x_\beta , f_\beta )\}_{\beta<\kappa}$ of a WLD space $\X$ gives rise to a prototypical \emph{projectional resolution of the identity} (PRI), a commuting family of linear norm-$1$ projections $\{P_{\alpha}\}_{\alpha<\kappa}$ such that there is a continuous strictly increasing map $\kappa \to \kappa$, $\alpha \mapsto \theta_\alpha$ such that 
$|\alpha| \vee \omega = |\theta_\alpha|$ and the image of $P_\alpha$ is 
$[x_\beta \colon \beta < \theta_\alpha]$ for each $\alpha<\kappa$ (see \cite{Hajek_biortsyst} for a formal definition). In such a case the given M-basis is said to be \emph{subordinated} to the PRI.
We define a \emph{coarsening} of a PRI $\{P_\alpha \}_{\alpha<\kappa}$ on $\X$ to be a subsequence 
$\{P_{\alpha_\beta}\}_{\beta<\kappa}$ that is also a PRI on $\X$.
Let us summarize:
\[\text{WLD}\ \implies\ \exists\ \text{M-basis},\ \text{(C)},\ \exists\ \text{PRI},\ 1\dash\text{SCP}, (\sigma) .\] 

\subsubsection{Partition relations in Ramsey theory}
We will use the following partition relation due to Erd\"{o}s and Rado, see e.g. \cite{EMHR,Jech}.
Let $\kappa$, $\lambda$ and $\mu$ be cardinals. Given a set $X$ we denote 
$[X]^n = \{A\subset X\colon |A|=n \}$, $n<\omega$, and by $[X]^{<\omega}$ the set of all 
finite subsets of $X$. 

The relation $\kappa \to (\lambda, \mu)^n$ for a fixed  $0<n<\omega$ means that for any function $c\colon [\kappa]^{<\omega} \to \{0,1\}$ (or coloring) there is 
$K_0 \subset \kappa$, $|K_0|=\lambda$, or there is $K_1 \subset \kappa$, $|K_1 | = \mu$, such that 
$c([K_0 ]^n ) =\{0\}$ or $c([K_1 ]^n) =\{1\}$. For example,
$(2^{\aleph_0})^+ \to ((2^{\aleph_0})^+ , \aleph_1 )^2$, see \cite[p.77]{EMHR}.

Recall the definition of $\lambda$-Erd\"{o}s cardinal; this is the least 
cardinal $\kappa$ such that
\[\kappa \to (\lambda)_{2}^{<\omega},\]
that is, for any coloring $c\colon [\kappa]^{<\omega} \to \{0,1\}$ there is 
$K\subset \kappa$, $|K|=\lambda$, such that for each $n<\omega$ the function $c$ is contant 
on $[K ]^n$ (but the values may vary with $n$). 

It seems reasonable to investigate the following 'hybrid' case. In this paper 
\begin{equation}\label{eq: partition}
\kappa \to (\lambda , \mu)^{<\omega}_2
\end{equation}
is a shorthand notation for the statement that for any coloring $c\colon [\kappa]^{<\omega} \to \{0,1\}$ there are $K_0 , K_1 \subset \kappa$ such that $|K_0 | = \lambda$, $|K_1 | = \mu$ and for each $0<n<\omega$
\[c([K_0 ]^n ) =\{0\}\ \vee \ c([K_1 ]^n) =\{1\}.\]
 
Note that the special case $\kappa \to (\mu , \mu)^{<\omega}_2$ is a consequence of $\kappa \to (\mu )^{<\omega}_2$ by using $K_0 = K_1=K$.

\section{Some remarks on long Schauder basic sequences}

The following result is motivated by Problem 8. in \cite{Lopez-Abad: problems}.

\begin{theorem}
Let $\X$ be a Banach space and $\{x_\alpha \}_{\alpha<\kappa} \subset \S_\X$ a sequence
satisfying
\[[x_\alpha \colon \alpha\leq \mu] \cap [x_\alpha\colon \mu < \alpha <\kappa]= \{0\}\] 
for all $\mu < \kappa$. Suppose that there is a stationary subsequence 
$S=\{\alpha_\gamma \}_{\gamma<\kappa} \subset \kappa$ such that for each $\mu< \kappa$ there is a $\mathrm{club}$ subset $C_\mu \subset \kappa$ with $[x_{\alpha_\gamma} \colon \gamma \leq \mu] + [x_\alpha \colon \alpha \in C_\mu ]$ closed for all 
$\mu< \kappa$. Then $\{x_\alpha\}_{\alpha<\kappa}$ contains a subsequence $\{x_{\alpha_\sigma}\}_{\sigma<\kappa}$ which is a Schauder basic sequence. In particular, one may take above a sequence of points of a strong M-basis $\{(x_\alpha,f_\alpha )\}_{\alpha<\kappa}$. Moreover, it suffices above that sets $C_\mu$ merely agree on $S$ with a $\mathrm{club}$ set, instead of being $\mathrm{club}$ sets themselves. 
\end{theorem}
\begin{proof}
We will construct by transfinite recursion a subsequence of $S=\{\alpha_\gamma \}_{\gamma<\kappa}$ 
in the following manner. Let $C_\mu \subset \kappa$ be sets as in the assumptions agreeing on $S$ with a 
$\mathrm{club}$ set, thus let $\widehat{C}_\mu \subset \kappa$ be $\mathrm{club}$ sets with 
$\widehat{C}_\mu \cap S = C_\mu $.
 
Let $\alpha_{\gamma_0}=\alpha_0$. For $0<\mu< \kappa$ we put 
\[\alpha_{\gamma_\mu} = \bigwedge  \bigcap_{\nu<\mu}  S \cap 
C_{\alpha_{\gamma_\nu}}.\]
Indeed, this construction is sensible by the basic properties of $\mathrm{club}$ and stationary sets. 
By inspecting the sequence $\{\alpha_{\gamma_\mu}\}_{\mu<\kappa}$ we observe that 
\[[x_{\alpha_{\gamma_\nu}} \colon \nu \leq \mu] + [x_{\alpha_{\gamma_\nu}} \colon \mu < \nu < \kappa]\]
is closed for all $\mu<\kappa$.

Recall that for closed subspaces $E,F \subset \X$ with $E\cap F = \{0\}$ the sum $E + F$ is closed if and only if the angle between the spaces is positive (see \cite{Hajek_biortsyst}).

It follows that we may consider the above decomposition as a direct sum and thus there are continuous linear projections 
\[P_\mu \colon [x_{\alpha_{\gamma_\nu}} \colon \nu \leq \mu] \oplus [x_{\alpha_{\gamma_\nu}} \colon \mu < \nu < \kappa] \to [x_{\alpha_{\gamma_\nu}} \colon \nu \leq \mu] ,\quad  \mu < \kappa .\]

Observe that 
\[\liminf_{\mu \to \kappa} \|P_\mu \| =C \in [1,\infty)\]
according to the regularity of $\kappa$ and the fact that $\inf_{\mu < \nu< \kappa} \|P_\nu \| \in [1,\infty)$ for all $\mu < \kappa$.
Thus we may choose a subsequence $\{\mu_\theta\}_{\theta<\kappa}$ such that 
$\|P_{\mu_\theta} \| \leq C+1$ for all $\theta< \kappa$. Then it is easy to see that 
$\{x_{\alpha_{\gamma_{\mu_\theta}}}\}_{\theta<\kappa}$ defines a Schauder basic sequence.

Finally, recall the well-known property of strong M-bases, namely, that for each disjoint $\Lambda, \Gamma \subset \kappa$ one has that 
$[x_\alpha \colon \alpha\in \Lambda] + [x_\alpha \colon \alpha\in \Gamma]$ is closed, see \cite{Hajek_biortsyst}.
\end{proof}

In the following result we consider basic sequences having possibly a special property: 
reverse monotone, modular unconditional, $1$-suppression unconditional, $1$-unconditional.  One may alternatively use the versions up to constant $C$ (e.g. projection constants uniformly bounded by $C$). By a \emph{transfinite block basic sequence} 
(t.b.b.s.) of a 
sequence $\{x_\alpha \}_{\alpha<\kappa}\subset \X$ we mean a long basic sequence of the form
\[u_\gamma = \sum_{\eta_{-}^{(\gamma)}\leq \sigma < \eta_{+}^{(\gamma)}} a_\sigma x_{\alpha_{\sigma}},\quad 
\gamma<\kappa\]
where $a_\sigma \in \R$, $\eta_{-}^{(\gamma_1)} < \eta_{+}^{(\gamma_1)} \leq \eta_{-}^{(\gamma_2)}<\eta_{+}^{(\gamma_2)} <\kappa$
for $\gamma_1 < \gamma_2 <\kappa$ and the summation is defined in a standard way by transfinite recursion with convergence in the order topology to norm topology sense.
 
\begin{theorem}
Let $\X$ be a Banach space with a Schauder basis $\{e_\alpha \}_{\alpha<\kappa}\subset \X$. 
\begin{enumerate}
\item{We have the following schema; suppose that 
each t.b.b.s. of $\{e_\alpha \}_{\alpha<\kappa}$ admits a further t.b.b.s. which has property (P), any well-defined property of transfinite basic sequences of order type $\kappa$ (e.g. symmetric). Then $\{z_\beta \}_{\beta<\kappa}\subset \S_\X$, $z_\beta \stackrel{\mathrm{w}}{\longrightarrow} 0$, admits a t.b.b.s. $\{u_\gamma \}_{\gamma<\kappa}$ with property (P). }
\item{Assume that $\X$ has Corson's property (C) and $\{e_\alpha \}_{\alpha<\kappa}$ has some of the properties listed above. Suppose that $\{z_\beta \}_{\beta<\kappa}\subset \X$ is a $\SD$ sequence (e.g. a weakly null or an M-basic sequence). Then there is a monotone basic sequence $\{z_{\beta_{\sigma}}\}_{\sigma < \kappa}$ having the same listed properties 
as $\{e_\alpha \}_{\alpha<\kappa}$. }
\end{enumerate}
\end{theorem}

We do not know if in the last part of the statement  one may consider some kind of transfinite spreading model approach, in place of a double t.b.b.s.

\begin{proof}
Let $\{z_\beta \}_{\beta<\kappa}\subset \X$ be as in the second part of the statement.
According to the Corson property and the considerations in \cite{Tal2} we obtain that $\{z_\beta \}_{\beta<\kappa}\subset \X$ converges weakly to $0$ and contains a subsequence $\{z_{\beta_{\sigma}}\}_{\sigma < \kappa}$ which is a monotone basic sequence. Moreover, by using the weak convergence and the basis coefficient functionals
of $\{e_\alpha \}_{\alpha<\kappa}$, we can choose the subsequence in such a way that the supports 
with respect to the basis satisfy $\bigvee \mathrm{supp} (z_{\beta_{\sigma_1}}) < \bigwedge \mathrm{supp} (z_{\beta_{\sigma_2}})$ for $\sigma_1 < \sigma_2 < \kappa$. Indeed, here we apply the regularity of $\kappa$. Now the verification of the first part of the statement is straight-forward.

Next we turn to the schema involving t.b.b.s.:s. 
Fix a t.b.b.s. of $\{z_{\beta_{\sigma}}\}_{\sigma < \kappa}$,
$u_\gamma = \sum_{\eta_{-}^{(\gamma)}\leq \sigma < \eta_{+}^{(\gamma)}} a_\sigma z_{\beta_{\sigma}}$,
$\gamma<\kappa$. By virtue of the weak convergence $z_\beta \stackrel{\mathrm{w}}{\longrightarrow} 0$ we may pass on to a subsequence $\{u_{\gamma_\theta} \}_{\theta<\kappa}$ which is a t.b.b.s. of 
$\{e_\alpha \}_{\alpha<\kappa}$. By the assumptions there is a further t.b.b.s.
$\{v_\gamma \}_{\gamma<\kappa}$ of $\{u_{\gamma_\theta} \}_{\theta<\kappa}$ which has property (P). Clearly, $\{v_\gamma \}_{\gamma<\kappa}$ is a t.b.b.s. of 
$\{z_\beta \}_{\beta < \kappa}$ as well.
\end{proof}
Compare the norming annihilator condition below to Proposition \ref{prop: bi_embed}.
\begin{proposition}
Suppose that $\X$ is a non-separable Banach space satisfying $(\sigma)$ and $\Y^\bot$ $1$-norms
a coseparable subspace of $\X$ for each separable subspace $\Y\subset \X$. Then $\X$ has a bimonotone basic sequence of length $\omega_1$. Moreover, each M-basis of $\X$ (provided that exists) has $\omega_1$-many mutually disjoint countable blocks that support such a basic sequence. If $\dens(\X)=\omega_1$ then the supports can be chosen to be successive in any (ex-ante) order type $\omega_1$ well-ordering of the M-basis. 
\end{proposition}
\begin{proof}[Sketch of Proof]
In the transfinite construction the crucial step is observing that if $\{y_\alpha \}_{\alpha < \theta} \subset \X$, $\theta<\omega_1$, is a bimonotone basic sequence, then there is a countable family $\mathcal{F} \subset \X^*$ which $1$-norms $[y_\alpha \colon \alpha < \theta]$. According to the assumptions there is 
a coseparable subspace $W \subset \X$ $1$-normed by $[y_\alpha \colon \alpha < \theta]^\bot$.
By using the $(\sigma)$ condition of $\X$ we can verify that 
$W\cap \bigcap_{f\in  \mathcal{F}} \ker f$ is coseparable as well and then choosing $y_\theta$ in this set
results in the required construction.   
\end{proof}
 
\section{Mixed Erd\"{o}s cardinals and long unconditional basic sequences}

The following combinatorics-driven analysis owes to the considerations in \cite{DLT, ketonen, LoTo2}
where countable unconditional basic sequences are extracted from long weakly null sequences of normalized vectors in Banach spaces. Here will study instead the extraction of \emph{uncountable} unconditional basic sequences. The weakly null sequences are required to be of high cardinality because the core of the argument is a purely combinatorial one. 

We will apply special cardinal numbers, such as Erd\"{o}s cardinals, originating from the infinitary combinatorics. Consider a smallest possible a cardinal $\kappa$ which satisfies $\kappa \to (\lambda, \aleph_0 )^{<\omega}_2$. This is clearly bounded from above by $\lambda$-Erd\"{o}s cardinal, assuming the relative consistency of the latter with the ZFC (see \cite{EMHR, Jech}). Presumably, such a cardinal $\kappa$ can be considered a large cardinal as well but we do not know reasonable lower bounds; for instance what is the relationship between statements such as
$\kappa \to (\aleph_1 ,\aleph_0 )^{<\omega}_2$ and $\kappa \to (\aleph_0 )^{<\omega}_2$.

The following result is a cleen consequence of a more general result which we will give subsequently
(see Theorem \ref{thm: comb}).
\begin{corollary}{$(\kappa \to (\lambda, \aleph_0 )^{<\omega}_2)$}\label{cor: ramsey}
\ \\
\noindent Let $\X$ be a Banach space, $\varepsilon>0$ a real, and $\kappa$, $\lambda$ be infinite cardinals as in the set theoretic assumption appearing above.
Then each normalized weakly null sequence $\{x_\alpha \}_{\alpha<\kappa} \subset \X$ contains 
a $1+\varepsilon$-suppression unconditional subsequence of length $\lambda$.
\end{corollary}

The existence of a $\lambda$-Erd\"{o}s cardinal, which implies that of the mixed Erd\"{o}s cardinal discussed above, has a rather high consistency strength.
However, the following combinatorial set-theoretic assumption appears in turn significantly weaker than the one employed above in Corollary \ref{cor: ramsey}.

\begin{theorem}\label{thm: comb}
Let us make the following set-theoretic assumption: Suppose that there exist infinite cardinals $\kappa$ and $\lambda$ such that for each coloring $c\colon [\kappa]^{<\omega} \to \{0,1\}$ there are 
subsets $K_0 , K_1 \subset \kappa$, $|K_0|=\aleph_0$, $|K_1|=\lambda$, such that at least one of the following conditions hold:
\begin{enumerate}
\item[(i)]{$(c\restriction [K_1 ]^{<\omega})^{-1} (\{1\}) \subset [K_1 ]^{<\omega}$ is cofinal
(partial order given by inclusion),}
\item[(ii)]{There exists $1<n<\omega$ such that $c([K_0 ]^n )=\{0\}$.} 
\end{enumerate}
Then each normalized weakly null sequence $\{x_\alpha \}_{\alpha<\kappa}$ of a Banach space contains a $1+\varepsilon$-suppression unconditional subsequence of length $\lambda$.
\end{theorem}
\begin{proof}
By using the weak-null assumption, possibly passing to a (weak-null) subsequence $\{x_{\alpha_\beta} \}_{\beta<\kappa}$, we may assume without loss of generality that no element occurs in $2$ different places in the sequence.
Let us study the set $[\{x_\alpha \colon \alpha<\kappa\}]^{<\omega}$.
In this collection we color all finite-length sequences, $c\colon [\{x_\alpha \colon \alpha<\kappa\}]^{<\omega} \to \{0,1\}$ as follows: All (finite-length) $1+\varepsilon$-unconditional basic sequences $\mapsto 0$
and other sequences $\mapsto 1$. We proceed in $2$ cases.

{\noindent \it Case} $\text{(i)}$: Suppose that there is a subfamily $\{x_{\alpha}\}_{\alpha \in K_1 }$, $K_1 \subset \kappa$, $|K_1 |=\lambda$, such that $(c\restriction [K_1 ]^{<\omega})^{-1} (\{1\}) \subset [K_1 ]^{<\omega}$ is cofinal. Let $L \subset K_1$ be a finite subset. Then there is 
a finite subset $M\subset K_1$, $L\subset M$, with $c(M)=1$. Thus $\{x_\alpha \}_{\alpha\in M}$ is 
$1+\varepsilon$-unconditional and hence $\{x_\alpha \}_{\alpha\in L}$ is such as well.
Then it is easy to see that $\{x_{\alpha}\}_{\alpha \in K_1 }$ is in fact a $1+\varepsilon$-unconditional basic sequence.

{\noindent \it Case} $\text{(ii)}$: According to the combinatorial principle in the assumptions, we are only required to exclude the second case. So, assume to the contrary that there is a subset 
$K_0 \subset \kappa$, $|K_0 | =\omega$ and $1<n<\omega$ with $c([K_0 ]^n )=\{0\}$. 

Let us consider $K_0$ with the inherited well-order and without loss of generality the order type of it 
is $\omega$. It is easy to see that $\{x_\alpha \}_{\alpha\in K_0}$ is weakly null as well. Finally, we invoke the fact that each weakly null normalized order type $\omega$ sequence of a Banach space contains for each $k < \omega$ and $\varepsilon >0 $ a $1+\varepsilon$-suppression unconditional basic sequence of length $k$, see \cite{DLT}. Thus there is $\{\alpha_{i}\}_{i<n}\subset K_0$ such that 
$\{x_{\alpha_i}\}_{i<n}$ is $1+\varepsilon$-suppression unconditional, so that 
$c ( \{x_{\alpha_{i}}\colon i<n\})=1$. This contradiction excludes the case (i).

By using the relative well-order of $L$ and the fact $|L|=\lambda$ and by disregarding 
suitably a final segment, we may pick the required subsequence having order type $\lambda$. 
\end{proof}

In the above argument the existence of long unconditional subsequence appears to be  
`countable-finite determined', or a rather strong form of separable reduction.
An adaptation of the above argument yields uncountable but non-large length sequences with 
lengthy subsequences having  a peculiar\footnote{The author thanks 
Th. Schlumprecht who asked in Maresias, Brazil, August 2014 meeting if this type of result holds, albeit 
in a slightly different setting.} 
partial unconditionality property (cf. \cite{LoTo}).  
\begin{theorem}\label{thm: combinatorial_b}
Let $\X$ be a Banach space, $\varepsilon>0$ a real, and $\kappa$, $\lambda$ be any infinite cardinals
and $1<n<\omega$. Suppose that $\kappa \to (\lambda, \aleph_0 )^{n}_2$ holds.
Then each normalized weakly null sequence $\{x_\alpha \}_{\alpha<\kappa} \subset \X$ contains 
a subsequence of length $\lambda$ which is a monotone basic sequence and such that each of its finite subsequences of length $n$ is $1+\varepsilon$-suppression unconditional. 
\end{theorem}

For example, according to the Erd\"{o}s-Rado theorem one may set 
$\kappa= \beth_{n-1}^+$ and $\lambda = \aleph_1$ above, 
see \cite[p.100]{EMHR} (but in general neither $\kappa$ nor $\lambda$ are required to have uncountable cofinality).

\begin{proof}[Sketch of Proof of Theorem \ref{thm: combinatorial_b}]
Note that the order of the sequence sought after is only relevant with regard to the monotonicity of the 
basic sequence, since $1+\varepsilon$-suppression unconditionality is defined irrespective of the ordering. 
First we prove the existence of a suitable subset of $L\subset \kappa$, $|L|=\lambda$, such that each finite subsequence of $\{x_\alpha \}_{\alpha\in L}$ of length $n$ is $1+\varepsilon$-suppression unconditional. 
The argument for this step is a straight-forward adaptation of the proof of Theorem \ref{thm: comb}, and 
we may assume without loss of generality that $L$ has order type $\lambda$. 

To arrange a monotone basic subsequence $\{x_{\pi(\theta)}\}_{\theta<\lambda}$ we use transfinite 
recursion. Pick $\pi(0)\in L$. Suppose that we have defined such a sequence up to 
$\{x_{\pi(\theta)}\}_{\theta<\gamma}$, $\gamma<\lambda$. Let $\mathcal{F}\subset \X^*$, 
$|\mathcal{F}|=|\gamma|\vee \omega$, be a $1$-norming subset for 
$[x_{\pi(\theta)}\colon \theta<\gamma]$. Then, using the weak null assumption and $|\gamma|<\lambda$, we observe that
\[\{x_\alpha\colon \alpha \in L,\ \forall f\in \mathcal{F}\ (f(x_\alpha )=0)\}\neq \emptyset\]
and we select $x_{\pi(\gamma)}$ from this set. By using simple cardinal arithmetic we see that this process will not terminate before $\lambda$ many steps. Finally, similarly as above, we may pass on to a subsequence of the required order type.
\end{proof}

\section{Games and freeness}
We are frequently interested here in subspaces of Banach spaces $1$-normed by subspaces of the dual, so let us fix the following notation: Let $\Psi \colon 2^{\X^* } \to 2^{\X}$ be
\begin{equation}\label{eq: Psi}
\Psi(U)=\left\{x\in \X\colon \sup_{f\in U\setminus \{0\}}\frac{|f(x)|}{\|f\|}=\|x\|\right\},\quad \Psi(\emptyset)=\Psi(\{0\})=\{0\}.
\end{equation}
Note that $\Psi(E) \subset \X$ is seldom a linear subspace even if $E \subset \X^*$ is such.
If $\X$ is a uniformly convex space and $E=P^* (\X^* )$ where $P \colon \X \to \Y$ is a linear
norm-$1$ projection onto, then $\Psi(E)=\Y$ but in general the (pre)duality mapping is far away from being convex (in fact the convexity of the duality mapping characterizes Hilbert spaces \cite{talponen_directional}).

However, $\Psi(E) \subset \X$ may easily contain closed subspaces. Thus, we are interested here in kind of \emph{spaceability} properties of the functor $\Psi$, cf. \cite{spaceable}. Further on, if $E$ is large (in some sense) then it seems reasonable to expect to find large closed subspaces in $\Psi(E)$. Here the largeness of the 
subspaces will be understood in the spirit of descriptive set theory, i.e. involving Baire category type considerations or winning strategies of topological games. 
It is a rather natural idea to use suitable games to analyze unconditional structures of Banach spaces, 
see e.g. \cite{JohnsonZ2}, \cite{LoTo2}, \cite{rosendal}, \cite{odell}, \cite{pelczar}.
Owing to ideas in Gowers' dichotomy (see \cite{Gowers} cf. \cite{maurey}) we first define a topological $1$-player game. The player chooses successively pairs; at stage $\alpha<\kappa$ she chooses a pair $(x_\alpha, Z_\alpha )$ where $x_\alpha \in \S_\X$ and $Z_\alpha \subset \X$ is a closed linear subspace. There are some rules; the player must 
\begin{enumerate}
\item{choose $Z_\beta \subset Z_\alpha$, $\alpha<\beta$,} 
\item{choose $x_\alpha \in \bigcap_{\gamma<\alpha} Z_\gamma$,}
\item{satisfy $[x_\gamma\colon \gamma<\alpha]+Z_\alpha \subset \Psi([x_\alpha]^\bot )$.}
\end{enumerate}
If the player cannot make an inning at a stage $\alpha<\kappa$ then the game ends, otherwise the game ends after $\kappa$-many stages. If the player can play $\kappa$-many innings, then she wins, otherwise she loses.

\begin{proposition}\label{prop:  single}
A Banach space $\X$ admits a $1$-suppression unconditional basic sequence of length $\kappa$ if and only 
if the player has a winning strategy on $\X$.
\end{proposition}
\begin{proof}
The main point is that $1$-suppression unconditionality of a sequence $(x_\alpha )_{\alpha<\mu}$ 
is characterized by the condition that $[x_\beta ]^\bot$ $1$-norms 
$[x_\alpha \colon \alpha\neq \beta]$ for any $\beta<\mu$.
\end{proof}

\newcommand{\PI}{\mathrm{P}_\mathrm{I}}
\renewcommand{\PII}{\mathrm{P}_\mathrm{II}}
\newcommand{\G}{\mathrm{G}}

The above game can be adapted to a two-player topological game $\G_{\X,\kappa}$ where the above player 
has the role of player $\mathrm{I}$ ($\PI$), player $\mathrm{II}$ ($\PII$, the spoiler)  simultaneously chooses points $y_\alpha \in \X$ at each stage $\alpha$ and the rule 
$(3)$ is replaced with
\begin{enumerate}
\item[($3^\prime$)]{$[x_\gamma,y_\gamma \colon \gamma<\alpha]+Z_\alpha \subset \Psi([x_\alpha]^\bot )$,}
\end{enumerate}
binding the choices of $\PI$. The game terminates at a stage $\alpha<\kappa$ if $\PI$ cannot make an inning
(following the rules).
Player $\PII$ wins if the game terminates at a stage $\alpha< \kappa$, otherwise $\PI$ wins. A strategy 
for $\PI$ is a mapping $s\colon (\{x_\alpha\}_{\alpha<\theta}, \{y_\alpha\}_{\alpha<\theta},
\{U_\alpha\}_{\alpha<\theta}) \mapsto (x_\theta ,U_\theta )$ consistent with the above rules.
We denote by $\PI \uparrow\G_{\X,\kappa}$ the fact that $\PI$ has a winning strategy in $G_{\X,\kappa}$.

Let us say that $x_\mu \in \S_\X$ is \emph{compatible} with a sequence 
$\{x_\alpha \}_{\alpha<\mu}\subset \X$ if 
\[[x_\alpha \colon \alpha<\mu ] \subset \Psi([x_\mu ]^\bot ) .\]
The possibility of compatible extensions in recursive constructions is essentially equivalent to the existence of a monotone basic sequence, a necessary but not by far a sufficient condition in the construction of 
unconditional basic sequences. 

\begin{lemma}\label{prop: compatible}
Suppose that $\X$ satisfies $\wst\dash\dens (\X^*)=\dens (\X)=\kappa$ and whenever $z_\mu$ is compatible with $\{z_\alpha \}_{\alpha<\mu}$, $\mu\!<\!\kappa$, then there is a closed linear space $\Z \subset \Psi([z_\mu ]^\bot )$ such that $\{z_\alpha \}_{\alpha<\mu} \subset \Z$ and $\codens(\Z)\leq |\mu|$. Then $\PI \uparrow\G_{\X,\kappa}$. This in turn implies that $\X$ has a $1$-suppression unconditional basic sequence of length $\kappa$.
\end{lemma}
\begin{proof}
The extension from a compatible case with transfinite recursion produces the required points $z_\mu$ 
where we collect $x_\gamma$, $y_\gamma$ specified in the game to obtain $z_\alpha$,  $\alpha<\mu$. 
Here $[x_\gamma , y_\gamma \colon \gamma<\mu]$ is $1$-normed by a suitable set 
\[\mathcal{F}_\mu \subset \X^* \]
with $|\mathcal{F}_\mu |\leq |\mu|$. Then 
\[\bigcap_{\alpha<\mu} \bigcap_{f\in \mathcal{F}_\alpha} \ker f \neq \{0\}\]
by the weak-star density assumption and this intersection consists of compatible points, so that the recursion 
proceeds for successor ordinals.

We claim that the recursion does not terminate before $\kappa$, i.e. the subspaces $\Z_\alpha$ produced in the
recursion satisfy $\bigcap_{\alpha<\mu} \Z_\alpha \neq \{0\}$ for $\mu<\kappa$. Indeed, assume to the 
contrary, then $\overline{\bigcup_{\alpha<\mu} \Z_{\alpha}^\bot }^{\wst} = \X^*$, contradicting the 
weak-star density assumption of the dual, since $\wst \dash \dens (\Z_{\alpha}^\bot)\leq |\mu|$, $\alpha<\mu$.

The latter part of the statement follows from Proposition \ref{prop:  single} since the existence of $\PI$ winning strategy in the $2$-player game is stronger than in the $1$-player version of the game.
\end{proof}

\begin{theorem}
Suppose that $\X$ is a weak Asplund space and $\PI \uparrow\G_{\X,\kappa}$. 
Then for any nested sequence of closed subspaces $\Y_\alpha \subset \X$ with 
$\codens(\Y_\alpha )<\kappa$, $\alpha<\kappa$, there is a $1$-suppression unconditional basic sequence 
$(x_\gamma )_{\gamma<\kappa}$ eventually included in any of the above subspaces.
Moreover, if $\X$ additionally has Corson's property (C) or the tightness of $(\X^* , \wst)$ is $< \kappa$, then $\wst\dash\dens (\X^*)\geq \kappa$.
\end{theorem}
\begin{proof}
Suppose that $\PII$ chooses points $y_0 , \ldots , y_\mu ,\ldots \in \S_\X$, $\mu< \kappa$, in such a way 
that for each $\alpha < \kappa$ there is $\sigma_\alpha < \kappa$ with $|\alpha|\vee \omega = |\sigma_\alpha| \vee \omega$ such that 
$\S_{\X/ \Y_{\alpha}} \subset q_{\X \to \X/ \Y_{\alpha}} ( \overline{\{y_\mu\colon \mu\leq \sigma_{\alpha}\}})$. 
We claim that the $\PI$ innings then satisfy that $x_\gamma \in \Y_\alpha$ for sufficiently large $\gamma$.

Indeed, because in a weak Asplund space the set of Gateaux smooth points is dense, we may run a closing-in argument to choose a sequence of smooth points $y_{0}' , \ldots , y_{\mu}' ,\ldots \in\S_\X$, $\mu<\kappa$, such that 
$\overline{\{y_\mu \colon \mu < \lambda\}} \subset  \overline{\{y_{\mu}' \colon \mu < \lambda\}}$ for ordinals of the form $\lambda = \lambda \omega < \kappa$ (under ordinal arithmetic). This way we may assume without loss of generality that the points $y_\mu$ are Gateaux-smooth in the first place.

By the Smulyan lemma it is easy to see that the set of points compatible with any of the points $y_{\mu}$ is in fact $\ker f_\mu$ where the functional is the unique support functional of $y_{\mu}$. Since the
functionals $f_\mu$, $\mu<\sigma_\alpha$, separate $\X/\Y_{\alpha}$, we see that 
$\bigcap_{\mu<\sigma_\alpha} \ker f_{\mu} \subset \Y_{\alpha}$. This shows that 
$x_\gamma \in \Y_{\alpha}$ for sufficiently large ordinals $\gamma <\kappa$.  
The fact that $(x_\gamma )_{\gamma<\kappa}$ resulting from the innings of $\PI$ is $1$-unconditional is essentially contained in Proposition \ref{prop:  single}. 

The statement regarding the density is seen similarly following the argument above and using Corson's property (C), which ensures that each $f \in \X^*$ is in the weak-star closed linear span of countably
many unique norm-attaining functionals corresponding to smooth points. Thus, if 
there were a weak-star dense set $\{g_\nu \colon \nu < \lambda\}\subset \X^*$, $\lambda < \kappa$, 
then there would also be a family of unique norm-attaining functionals of the same cardinality $<\kappa$ which would separate $\X$. Then $\PII$ would win the game by playing the corresponding smooth points. 
\end{proof}

In the above result we have \emph{tails} of long basic sequences included in given deficient codensity subspaces. This can be viewed as a kind of genericity of the class of long $1$-suppression unconditional basic sequences in the space.
It seems reasonable to ask, along the lines in \cite{DLT, ketonen}, what are the 
smallest cardinals $\kappa_{\mu}$ such that $\PI \uparrow\G_{\X,\mu}$ for any 
Banach space $\X$ with $\dens(\X)\geq \kappa_\mu$.

\subsection{Asymptotic freeness}
Suppose that $\Z\subset \X$, $\dens(\Z)<\kappa$, is a closed subspace and $\{\Z_\alpha \}_{\alpha<\kappa} \subset \X$ is a nested sequence of closed subspaces with 
\[\Z_\alpha \stackrel{\w}{\longrightarrow} \Z\quad \text{as}\  \alpha\to\kappa,\] 
meaning $\Z^\bot = \bigcup_{\alpha<\kappa} \Z_\alpha ^\bot$. This implies that $\bigcap_{\alpha<\kappa} \Z_\alpha =\Z$ and in fact these conditions are equivalent 
if $\X$ has property (C), see \cite{Tal2}.

Intuitively speaking, we 
would like to call $\X$ `asymptotically free' if the satisfaction relation $\Z_{\alpha}\models$ is $\alpha$-continuous at $\kappa$; the satisfaction of sentences stabilizes in the way that the verity of a first order predicate logic sentence 
\[\forall z\in \Z \ (\phi (z))\]
(in the signature of $\X$) implies that there is $\alpha<\kappa$ such that
\[\forall z\in \Z_\beta \ (\phi (z)) \quad \text{for}\ \alpha<\beta <\kappa .\]
This approach (cf. \cite[Ch.5]{Iovino}) appears 
too general here, in the context of the geometry of Banach spaces, and therefore we restrict ourselves to a simpler form.
Given a cardinal $\lambda<\kappa$ we call $\X$ $(\lambda,\kappa)$-\emph{asymptotically free} ($(\lambda,\kappa)$-a.f.) if for all $x\in \X$, $\{\Z_\alpha\}_{\alpha <\kappa}$, and $\Z$ as above with $\dens(\Z)\leq \lambda$ we have
\[\forall z\in \Z\  (\|x+z\|\geq \|z\|)\]
implies that already for some $\alpha<\kappa$ it holds that
\[\forall z\in \Z_\alpha \ (\|x+z\|\geq \|z\|) .\]

Note that a space $\X$ with $\dens(\X) < \kappa$ is trivially $(\lambda,\kappa)$-a.f. and that this 
property is inherited by closed subspaces.
Let us call $\X$ simply $\kappa$-a.f. if it is $(\lambda , \kappa)$-a.f. for all $\lambda<\kappa$.
Also, straight-forward isomorphic generalization of this notion is possible, i.e. $(\lambda,\kappa)$-a.f. up to a renorming. There is a geometric type condition which guarantees asymptotic freeness. 
If  the following always holds with the above notations, then it is easy to see that the space is $(\lambda,\kappa)$-a.f.:
\[d((\S_{\Z_\alpha }+[x])\cap \B_\X , \S_{\X}) \to d((\S_{\Z}+[x])\cap \B_\X , \S_{\X}),
\quad \alpha\to \kappa .\]
Here $d$ is the non-symmetric Hausdorff distance.
\begin{theorem}\label{thm: spread}
A Banach space $\X$, $\dens(\X)=\kappa$, is $\kappa$-a.f. if it embeds linearly isometrically into a space $\mathrm{W}$ with a modular unconditional basis. Moreover, if $2^{\lambda} < \kappa\leq \mu$ then $\ell_{\lambda}^\infty (\mu)$ is $(\lambda,\kappa)$-a.f.
\end{theorem}

For example, under GCH the space $\ell_{\omega_\alpha}^\infty (\omega_{\alpha+2})$ is $(\omega_{\alpha},\omega_{\alpha+2} )$-a.f.
Note that the first part of the statement includes the space $c_0 (\omega_1 )$, roughly corresponding to the case with $<\!\!\omega$ and $ \omega_1$ in place of $\lambda$ and $\mu=\kappa$, respectively.

One might ask, whether the role of $\lambda$ and $\mu=\kappa$ is somewhat analogous to the same symbols in the partition relation \eqref{eq: partition}. It is clear that enlarging $\lambda$ makes the $(\lambda,\kappa)$-a.f. condition stronger. However, the effect of $\kappa$ works in reverse. The latter part of the above result provides us with some analytics on the mutual relationship of $\lambda$ and $\kappa$; it suggests that large Banach spaces are inclined to be asymptotically free of entanglement in large ($\kappa$) scales.  
See also Final Remarks.

\begin{proof}[Proof of Theorem \ref{thm: spread}]
Since a.f. is a hereditary property we need only to show that $\mathrm{W}$ and 
$\ell_{\lambda}^\infty (\mu)$ have the respective a.f. properties. We will first give the argument for 
the second part of the statement. 

Fix the band ideal $I \subset  \ell_{\lambda}^\infty (\mu)$ generated by the supports of given $x\in \ell_{\lambda}^\infty (\mu)$ and of $\Z \subset \ell_{\lambda}^\infty (\mu)$, $\dens(\Z)\leq \lambda$,
as in the definition of asymptotic freeness. Let $P_I$ be the corresponding band projection.
Note that altogether $\leq\! |\dens(\Z) \times \lambda|=\lambda$ many coordinates are supported on $I$. 
Thus $I$ has density $\dens(I)\leq (\aleph_0 )^\lambda = 2^\lambda < \kappa$ 
by basic cardinal arithmetic and the assumptions. Suppose that 
$\Z_\alpha \stackrel{\w}{\longrightarrow} \Z $, $\alpha \to \kappa$, as in the definition of asymptotic freeness. This means that $\Z^\bot = \bigcup_{\alpha<\kappa} \Z_{\alpha}^\bot$. Thus, 
$P_{I}^* (\Z^\bot )= \bigcup_{\alpha<\kappa} P_{I}^* (\Z_{\alpha}^\bot )$. Hence 
$P_{I} (\Z_\alpha )\stackrel{\w}{\longrightarrow} P_{I} (\Z) $, $\alpha \to \kappa$, holds in the space $I$.
Since $\dens(I)<\kappa$ there is already $\beta <\kappa$ such that $P_{I} (\Z_\beta ) = P_{I} (\Z)$. 
Suppose that $x$ and $\Z$ satisfy $\forall z \in \Z$ $(\|x+z\| \geq \|z\|)$.
Then, by the selection of the band ideal $I$,
\[ \|P_I (x+z)\| = \|x+z\| \geq \|z\| = \|P_I (z)\|,\quad z \in \Z ,\]
and for $z\in \Z_\alpha$, $\beta < \alpha <\kappa$, we have 
\begin{equation}\label{eq: maxnorm}
\|x+z\| = \|P_I (x+z)\| \vee \|(\I-P_I ) (z)\| \geq  \|P_I (z)\| \vee \|(\I-P_I ) (z)\| 
=\|z\| .
\end{equation}
This proves the second statement of the theorem. 

The first statement is obtained as a modification of the above argument, mainly dispensing with the 
cardinal arithmetic, as the supports of vectors, in terms of the Schauder basis, are countable.
Because of the assumption $\dens(\X)=\kappa$ we may assume without loss of generality that 
the basis of $\mathrm{W}$ has length $\kappa$. Recall that the modularity of the basis implies it is $1$-suppression unconditional. One uses again a suitable band ideal $I$ and a corresponding projection $P_I$, 
which is available because the basis is unconditional. Then there is $\sigma < \kappa$ such that 
$I$ is supported on initial segment of the basis bounded by $\sigma$.
The modularity of the basis is finally applied in proving an inequality analogous to \eqref{eq: maxnorm}. 
\end{proof}

\begin{corollary}
Let $\kappa$ be a strongly inaccessible cardinal. Then $\ell_{<\kappa}^\infty (\kappa)$ is 
$\kappa$-a.f. \qed
\end{corollary}

The above notion heavily relies on the countable tightness properties of Banach space and the uncountable regular length of the sequences. Other notions of freeness in Banach spaces in the separable setting include property (u), asymptotic unconditionality and the Gordon-Lewis property 
(e.g. \cite{saab}, \cite{Cowell}, \cite{JohnsonZ}, \cite{GL}). Next we show how asymptotic freeness 
implies the genericity of long unconditional basic sequences in the space.

\begin{theorem}\label{thm: af}
Suppose that $\X$, $\dens(\X)\!= \!\kappa$, is a $\kappa$-a.f. (resp. $(\lambda,\kappa)$-a.f.) WLD space. 
Then $\PI \uparrow \G_{\X,\kappa}$ (resp. $\PI \uparrow \G_{\X,\lambda^+}$).
\end{theorem}
\begin{proof}
Let $\{(z_\alpha , f_\alpha )\}_{\alpha < \kappa}$ be an M-basis as in the characterization 
of WLD spaces \eqref{eq: count}. Fix $\mu<\kappa$ and suppose that innings $\Z_\alpha$, $y_\alpha$, 
$\alpha<\mu$, with $\codens(\Z_\alpha )<\kappa$ and points $x_\alpha$, $\alpha\leq \mu$, have been played in the game. In the second part of the statement we assume $\mu<\lambda^+$ instead but otherwise the argument runs similarly.
Here $x_\mu$ must be in particular compatible with 
$\{x_\alpha , y_\alpha \colon \alpha<\mu\}$. This can be arranged since $\wst\dash\dens(\X^*) = \kappa$. This in turn is due to $\dens(\X)=\kappa$ and WLD assumptions.
Let $\lambda_\mu <\kappa$ be such that 
\[x_\alpha, y_\beta \in [z_\gamma \colon \gamma< \lambda_\mu ],\quad \alpha\leq \mu,\ \beta<\mu.\]
By standard considerations in \cite{Tal2}, using Corson's property (C), we have that 
\begin{equation}\label{eq: setW}
\overline{[z_\gamma \colon \gamma< \lambda_\mu ] + \bigcap_{\nu<\eta} \ker f_\nu} \stackrel{\w}{\longrightarrow}  [z_\gamma \colon \gamma< \lambda_\mu ],\quad \eta\to \kappa .
\end{equation}

By using the compatibility condition we have that $\|z + x_\mu\| \geq \|z\|$
for all $z\in [x_\alpha, y_\alpha\colon \alpha<\mu]$.

Thus, by using \eqref{eq: setW} with the asymptotic freeness assumption we obtain that 
there is $\eta <\kappa$ such that 
$\|z + x_\mu\| \geq \|z\|$ holds for all
\[z\in \overline{[x_\alpha , y_\alpha \colon \alpha< \mu ] + \bigcap_{\nu <\eta } \ker f_\nu}.\]
It is easy to see that then in fact $\|z + r x_\mu\| \geq \|z\|$ holds for all $r\in \R$ and 
\[z\in \overline{[x_\alpha , y_\alpha \colon \alpha< \mu ] + \bigcap_{\nu <\eta } \ker f_\nu}.\]
A suitable subspace $\Z_\mu$ for the $\PI$ inning is then 
\[\Z_\mu := \overline{[x_\alpha , y_\beta \colon \alpha\leq \mu,\ \beta<\mu ] + \bigcap_{\nu<\eta} \ker f_\nu}\ \cap\ \bigcap_{\alpha<\mu} \Z_\alpha .\]
Note that $\codens(\Z_\mu) \leq \eta \vee \bigvee_{\alpha<\mu}\codens(\Z_\alpha ) < \kappa$. Thus the game can proceed towards $\PI$ winning.
\end{proof}

The following result has some bearing on a question posed  in \cite{DLT} regarding the value of $\mathfrak{nc}_{\mathrm{rfl}}$, the least cardinal $\kappa$ such that any reflexive space 
of density $\geq\! \kappa$ has a countable unconditional basic sequence. 
Perhaps $\mathfrak{c}^+ =(2^\omega )^+$ is a somewhat natural candidate for $\mathfrak{nc}_{\mathrm{rfl}}$ since $\mathfrak{nc}_{\mathrm{rfl}} > \aleph_1$ (\cite{ALT}), $\mathfrak{c} =\aleph_1$ consistently and $\mathfrak{nc}\geq \mathfrak{c}^+$ (\cite{AT2}).

\begin{corollary}\label{cor: only}
Let $\X$ be a WLD space with $\dens(\X)= \kappa \geq \mathfrak{c}^+$. Suppose that $\X$ is isomorphically 
a subspace of $\ell_{\omega}^\infty (\kappa)$, or, equivalently, there is $C>0$, and 
$\{f_\alpha \}_{\alpha< \kappa}\subset \S_{\X^*}$ such that $C\B_{\X^*} \subset \overline{\conv}^{\wst}(f_\alpha \colon \alpha< \kappa )$ and 
\[|\{\alpha< \kappa \colon f_\alpha (x) \neq 0 \}| \leq \aleph_0\quad \forall x \in \X .\] 
Then $\X$ has an unconditional basic sequence of length $\omega_1$.
In the case with isometric embedding the unconditional sequence can be chosen to be additionally $1$-suppression unconditional. 
\end{corollary}
\begin{proof}
In the case of $C=1$ Theorem \ref{thm: af} and and Lemma \ref{prop: compatible} apply directly. The case 
$C<1$ is then obtained by a straight-forward renorming argument.
\end{proof}

Above the isomorphic containment in $\ell_{\lambda}(\kappa)$ is not an assumption of topological nature 
as in the typical analysis of the weak topology of M-bases. Instead, the norming properties of sets of functionals are essential here. Recall that every Banach space $\X$ of density $\kappa$ can be isometrically embedded in $\ell^\infty (\kappa)$ but its subspace $\ell_{\omega}^\infty (\kappa)$
is considerably smaller, possibly with a smaller density. For example, infinite direct sums of separable spaces 
in the $c_0$-sense satisfy the hypothesis of the above result.

\section{Bimonotone PRIs and basic sequences}

Recall that the $(\sigma)$ property can be seen as a kind of Banach space version of the Baire property of topological spaces. Next we will in a sense dualize the $(\sigma)$ property to obtain a kind of contravariance principle, stating that the annihilator of a small subspace $1$-norms a large subspace. 
Compare to Proposition \ref{prop: bi_embed}.
\begin{enumerate}
\item[$\mathrm{(d\dash d)}$] Given $\Y \subset \X$, $\dens(\Y)<\dens(\X)$, there is a closed subspace $\Z \subset \X$ with 
$\Z \subset \Psi(\Y^\bot )$, $\codens(\Z)\leq \dens(\Y)$. 
\end{enumerate}

By a projectional sequence we mean a system of bounded linear projections as in PRI, except that the projection constants are only required to be uniformly bounded (instead of being $1$), see \cite{Hajek_biortsyst}.
By a bimonotone PRI we mean a PRI where the  projections are bicontractive, i.e. both the projections and the corresponding coprojections are norm-$1$. By the bicontractive Separable Complementatation Property we mean that every separable subspace is contained in a separable subspace which is complemented by a \emph{bicontractive} linear projection.

\begin{theorem}\label{thm: main_PRI}
Let $\X$, $\dens(\X)=\kappa$, be a WLD space satisfying condition $\mathrm{(d\dash d)}$.
Then the following conditions hold: 
\begin{enumerate}
\item{$\X$ has bicontractive SCP. Moreover, every subspace $\Y \subset \X$, $\dens(\Y) < \dens(\X)$, admits a bicontractively complemented space $\Z\subset \X$ with $\Y\subset \Z$ and $\dens(\Z)=\dens(\Y)$.}
\item{Any M-basis of $\X$ is, up to a permutation, subordinated to a bimonotone PRI on $\X$} 
\item{If $\dens(\X)=\aleph_1$, given any M-basis on $\X$, any PRI on $\X$ admits a bimonotone coarsening which the M-basis is subordinated to.}
\end{enumerate} 
\end{theorem}
\begin{proof}
As a WLD space $\X$ admits an M-basis $\{(x_\beta, f_\beta )\}_{\alpha<\kappa}$ which in turn can be used in constructing a PRI in this case. The first part of the statement is reduced to the second part by representing $\Y$ in the M-basis and permuting the M-basis in such a way that the support of $\Y$ is an initial segment of it. The bicontractive projection to a superspace of the space generated by the initial segment is constructed
in the inductive step of the simple-case argument of the second part of the statement. The crucial 
fact is that $|\bigcup_{n<\omega} \Theta_n | = \dens(\Y)$ can be arranged below.

To verify the second part of the statement let $\{(x_\beta, f_\beta )\}_{\alpha<\kappa}$ be an M-basis. According to \eqref{eq: count} for any infinite subset $\mathcal{F}\subset \X^{*}$ with $|\mathcal{F}|=\lambda$ there is $\Theta \subset \kappa$, $|\Theta|=\lambda$, such that 
$\{x_\beta \colon \beta \in \kappa \setminus \Theta\} \subset \bigcap_{f\in \mathcal{F}} \ker(f)$.
Then we define 
\[\Xi (\mathcal{F} ) = \bigwedge \Big\{\sigma<\kappa\colon \forall \beta \in [\sigma,\kappa)\ x_\beta \in 
\bigcap_{f\in \mathcal{F}} \ker(f) \Big\}.\]

To prove the statement recursively, assume that the Bimonotone PRI sequence 
$P_\alpha \colon \X \to [x_\beta \colon \beta <\theta_\alpha ]$ has been constructed up to $\alpha<\gamma$, where $\gamma<\kappa$, such that $|\theta_{\alpha+1}\setminus \theta_\alpha |\geq \aleph_0$ (possibly $\gamma=0$).

Thus $\big|\bigvee_{\alpha<\gamma} \theta_\alpha \big| \leq |\gamma| \vee \aleph_0$.
Let $\mathcal{F}_0 \subset \X^*$, $|\mathcal{F}_0 | =|\gamma| \vee \aleph_0$, be such that
$[x_\beta \colon \beta <  \omega\vee \bigvee_{\alpha<\gamma} \theta_\alpha] \subset \Psi (\mathcal{F}_0 )$ (recall \eqref{eq: Psi}).
According to $\mathrm{(d\dash d)}$ assumption there is $\Z_0 \subset \X$, $\Z_0 \subset \Psi ([x_\beta \colon \beta <  \omega\vee \bigvee_{\alpha<\gamma} \theta_\alpha]^\bot )$ with $\codens(\Z_0 ) \leq |\gamma| \vee \aleph_0$.
By using the generalized $(\sigma)$ type codensity estimates of WLD spaces as before, we obtain that 
$\codens(\Z_0 \cap \bigcap_{f \in \mathcal{F}_0 }\ker(f) )\leq |\gamma| \vee \aleph_0$.
Using the characterizing property \eqref{eq: count} of WLD spaces (and the subsequent remark)
there is a set $\Theta_0 \subset \kappa$, $|\Theta_0 | \leq  |\gamma| \vee \aleph_0$, 
such that $\Z_0 \cap \bigcap_{f \in \mathcal{F}_0 }\ker(f) \supset 
[x_\beta \colon \beta \in \kappa\setminus \Theta_0]$.

For the sake of clarity we proceed in $2$ cases. The first case, a partial result, does not require permutation of the M-basis, and the second case, which requires it, covers the full statement.

{\noindent \it Simple case:} We first assume that the M-basis is such that we may choose $\sigma_n$:s below
in such a way that $\Big|\sigma_0 \setminus  \bigvee_{\alpha<\gamma} \theta_\alpha \Big|=|\sigma_{n+1} \setminus \sigma_n | = |\gamma| \vee \aleph_0 $ for $n<\omega$.

We set $\sigma_0 = \Xi \big(\big(\Z_0 \cap \bigcap_{f \in \mathcal{F}_0 }\ker(f) \big)^\bot \big)$.
Next we define $\mathcal{F}_1$ and $\Z_1$ similarly as above replacing $\bigvee_{\alpha<\gamma} \theta_\alpha$ with $\sigma_0$. There is again an exceptional set  $\Theta_1 \subset [\sigma_0 ,\kappa)$, 
$|\Theta_1 | \leq  |\gamma| \vee \aleph_0$, such that $\Z_1 \cap \bigcap_{f \in \mathcal{F}_1 }\ker(f) \supset 
[x_\beta \colon \beta \in \kappa\setminus \Theta_1]$.
Then we set $\sigma_1 = \Xi \big(\big(\Z_1 \cap \bigcap_{f \in \mathcal{F}_1 }\ker(f) \big)^\bot \big)$. Then we obtain $\mathcal{F}_2$ and $\Z_2$ similarly by replacing 
$\sigma_0$ with $\sigma_1$. Similarly we obtain an exceptional set $\Theta_2 \subset [\sigma_1 , \kappa)$, $|\Theta_2 | \leq  |\gamma| \vee \aleph_0$ and define $\sigma_2 = \Xi \big(\big(\Z_2 \cap \bigcap_{f \in \mathcal{F}_2 }\ker(f) \big)^\bot \big)$. We proceed in this manner to obtain an increasing sequence
of ordinals $\sigma_0 < \sigma_1 < \ldots < \sigma_n < \ldots$, $n<\omega$. Recall the simplifying 
assumption that $|\sigma_{n+1} \setminus \sigma_n | \leq |\gamma| \vee \aleph_0$ for each $n<\omega$. Thus $|\bigvee_{n<\omega} \sigma_n | = |\gamma| \vee \aleph_0$.
Then setting $P_\gamma \colon \X \to [x_\beta \colon \beta < \bigvee_{n<\omega} \sigma_n]$,
$P_\gamma \colon x+ z \mapsto x$, $x\in  [x_\beta \colon \beta < \bigvee_{n<\omega} \sigma_n]$, 
$z\in  [x_\beta \colon  \bigvee_{n<\omega} \sigma_n \leq \beta < \kappa]$, defines a bicontractive 
projection. Indeed, the crucial observation here is that, given $x\in  \span(x_\beta \colon \beta < \bigvee_{n<\omega} \sigma_n )$ and $z\in  \span (x_\beta \colon  \bigvee_{n<\omega} \sigma_n \leq \beta < \kappa )$, there is $n<\omega$ such that $x \in  \span(x_\beta \colon \beta < \sigma_n )$ and
 $z \in  \span(x_\beta \colon \sigma_{n+1} \leq \beta <\kappa )$. Thus 
$z \in \bigcap_{f \in \mathcal{F}_{n+1} }\ker(f)$ and $z \in \Psi (\span(x_\beta \colon \beta < \sigma_n )^\bot )$, hence $\|x\|\leq \|x+z\| \geq \|z\|$. This means that in the recursive construction of a Bimonotone PRI the extension to $P_{\alpha+1 }$, $\alpha<\kappa$, is possible. By a similar argument as above with the linear spans the extension to limit ordinals $\lambda<\kappa$ follows; it is 
accomplished  formally by taking the closure of the union of graphs $\overline{\bigcup_{\alpha<\lambda} \Gamma P_\alpha } \subset \X \oplus \X$. It is then clear that the union of images of $P_\alpha$ 
for $\alpha<\lambda$ is dense in the image of $P_\lambda$ and has density $|\lambda|$.

{\noindent \it Full case:} In order to stabilize the argument above we had to make sure that the increases of $\sigma_n$:s have controlled cardinality. 
However, the order type of\\
$\bigwedge \Big\{\sigma < \kappa\colon [\sigma,\kappa ) \subset \kappa \setminus \bigcup_{n<\omega}\Theta_n \Big\}$
can be any ordinal $<\!\!\kappa$, even for finite $\Theta_n$:s. 
We will indicate the required changes to obtain the full statement with a suitable permutation of the 
M-basis.

Instead of $\sigma_{n}^{(\alpha)}$:s we will consider an increasing mapping $\varsigma$ with a modified definition. To introduce some other auxiliary mappings, consider $\kappa \times \omega$ in the lexicographic order and the power set $2^\kappa$ partially ordered by inclusion. 
We will construct a decreasing mapping
$F\colon \kappa \times \omega \to 2^\kappa$. This mapping satisfies 
$F(\alpha+1 , 0)=  \bigcap_{n<\omega} F(\alpha,n)$, $F(\lambda,0)=\bigcap_{\alpha<\lambda} F(\alpha,0)$ for limit ordinals $\lambda<\kappa$. Note that each of these subsets is well-ordered.
Intuitively speaking, the mapping $F$ is related to the $\alpha$:th inductive step of the above construction of the PRI as follows:  $F(\alpha,n)\setminus F(\alpha,n+1) = \Theta_{n+1}$ (up to some modifications). 

We construct bijections $\pi_{\alpha,n} \colon F(\alpha,n) \to F(\alpha,n+1)$ recursively. 
In this construction each ordinal is permuted at most once to another position (and in that case there is 
a unique $\pi_{\alpha,n}$ performing the change). Therefore we may take 
point-wise limits of the permutations; 
$\pi_{(\alpha+1,0)}\colon F(\alpha+1 ,0) \to F(\alpha+1 ,0)$, $\pi_{(\alpha+1,0)}:= \lim_{n\to\omega} \pi_{\alpha,n}$ and $\pi_{(\lambda,0)} \colon F(\lambda,0)\to F(\lambda,0)$, $\pi_{(\lambda,0)} := \lim_{\alpha\to\lambda} \pi_{\alpha,0}$ for a limit ordinal $\lambda<\kappa$.
Here the convergence of the nets is with respect to the product topology on $\kappa^\kappa$ where $\kappa$ is considered in the \emph{discrete} topology.
Let $F(0,0)=\kappa$, $\varsigma (0,0)=0$, $\pi_{0,0} = \mathrm{Id} \colon \kappa\to\kappa$.
Suppose that we have defined a bijection
$\pi_{\alpha,n} \colon F(\alpha,n) \to F(\alpha,n+1)$, $n<\omega$, and (non-permutable control ordinal) 
$\varsigma(\alpha,n) < \kappa$. We impose the following conditions: 
\begin{enumerate}
\item $F(\alpha,n) \setminus F(\alpha,n+1) \subset [\varsigma(\alpha,n), \kappa)$,\quad $|F(\alpha,n) \setminus F(\alpha,n+1)| = |\alpha|\vee \aleph_0$;
\item $\varsigma$ is increasing and moreover
\[|(\varsigma(\alpha,n+1)\cap F(\alpha,n+1)) \setminus (\varsigma(\alpha,n)\cap F(\alpha,n))|=
|\alpha|\vee \aleph_0 ;\]
\item[(3-4)]
\[[x_\beta \colon \beta \in [\varsigma(\alpha,n+1) ,\kappa) \cap\pi_{\alpha,n}( F(\alpha,n+1)) ]
\subset \Psi ([x_\beta \colon \beta \in  \varsigma(\alpha,n)\cap\pi_{\alpha,n}(F(\alpha,n) )]^\bot);\]
\[[x_\beta \colon \beta \in \varsigma(\alpha,n) \cap \pi_{\alpha,n} ( F(\alpha,n) )] \subset 
\Psi([x_\beta \colon \beta \in [\varsigma(\alpha,n+1) ,\kappa) \cap\pi_{\alpha,n} ( F(\alpha,n+1) )]^\bot );\]
\item[(5)] $\pi_{\alpha,n+1} (F(\alpha,n)\setminus F(\alpha,n+1)) \subset [\varsigma(\alpha,n), \varsigma(\alpha,n+1))$;
\item[(6)] $\pi_{\alpha,n+1}\restriction (F(\alpha,n+1) \setminus [\varsigma(\alpha,n) , \varsigma(\alpha,n+1)))
=\mathrm{Id}\restriction$ ;
\item[(7)] {$\pi_{\alpha,n} \restriction F(\alpha,n)\cap \varsigma(\alpha,n)=\pi_{\beta,m} \restriction F(\alpha,n) \cap \varsigma(\alpha,n)$ for $(\alpha,n)\leq (\beta,m)$.}
\end{enumerate}
We define recursively bijections $\pi^{(\alpha)} \colon \kappa \to F(\alpha,0)$ by setting $\pi^{(0)}=\pi_{0,0}$
and 
$\pi^{(\alpha+1)}=\lim_{n\to\omega} (\pi_{\alpha, n} \circ \dots \circ \pi_{\alpha,1}\circ  \pi_{\alpha,0}
\circ \pi^{(\alpha)})$  and 
$\pi^{(\lambda)}=\lim_{\alpha\to\lambda} \pi^{(\alpha)}$ for limit ordinals $\lambda\leq \kappa$;
the limits are again taken in the sense of point-wise convergence.

Now, $\pi^{(\kappa)} \colon \kappa \to F(\kappa ,0 ) \subset \kappa$ is a bijection.
Thus $\pi^{(\kappa)} [0,\kappa)$ is a well-ordered set order isomorphic to $\kappa$ by the regularity of 
$\kappa$. The required permutation is given by $\{y_\gamma \}_{\gamma<\kappa}$ with $y_\gamma = x_{\eta}$, such that $\pi^{(\kappa)}(\eta)\in\pi^{(\kappa)} [0,\kappa)$ has order type $\gamma$ relative to $F(\kappa,0)$. The required PRI projections have the form 
$P_\alpha \colon \X \to [x_\beta \colon \beta \in \pi_{\alpha,0} (F(\alpha,0))\cap \varsigma(\alpha,0)]$
where the indexing begins with $\alpha=1$.

To verify the last part of the statement, let $\{P_\alpha \}_{\alpha<\omega_1}$ be a PRI on $\X$. 
The argument is a modification of the proof of the second part of the statement restricted to the simple case.
Let $\{(x_\beta, f_\beta )\}_{\alpha<\omega_1}$ be an M-basis of $\X$.
Denote
\[\sigma_\alpha = \Xi (\mathcal{F} ),\quad \overline{[\mathcal{F}]}^{\wst}=P_{\alpha}^* (\X^* )\subset \X^* ,\
\alpha< \omega_1 \]
(this definition does not depend on the particular choice of $\mathcal{F}$). Suppose that we have constructed
the required PRI up to $\gamma<\omega_1$, i.e. sequences 
$\{\beta_{\eta}\}_{\eta<\gamma}, \{\alpha_{\eta}\}_{\eta<\gamma} \subset \omega_1$ such that 
\[P_{\eta}' = P_{\alpha_\eta} \colon \X \to [x_\beta\colon \beta < \beta_\eta]\]
define the required bicontractive projections for all $\eta<\gamma$ (possibly $\gamma=0$).
If $\gamma$ is a limit ordinal, then it is easy to see, similarly as above, that we may take 
the closure of the graphs to extend to a bicontractive projection\\ 
$P_{\gamma}' \colon \X = \left[x_\beta \colon \beta < \bigvee_{\eta<\gamma} \beta_\eta\right] \oplus 
\left[x_\beta \colon \bigvee_{\eta<\gamma} \beta_\eta \leq \beta < \omega_1\right]\to \left[ x_\beta \colon \beta < \bigvee_{\eta<\gamma} \beta_\eta\right]$.
The inductive step is defined as follows: Let $\epsilon_0 = \left(\bigvee_{\eta<\gamma} \beta_\eta \right) +\omega$.
Fix a family $\mathcal{F}_0 \subset \X^*$ with
$\left[x_\beta \colon \beta< \epsilon_0  \right] \subset \Psi (\mathcal{F}_0 )$. By using the basic properties of WLD spaces let $\phi_0 = \Xi (\mathcal{F}_0 ) <\omega_1$.
Then necessarily $\epsilon_0 < \phi_0$.
By using the $(\sigma)$ property of $\X$ we can again find the least $\sigma_0 < \omega_1$ such that 
$\phi_0 < \sigma_0$ and 
$\left[x_\beta \colon \beta< \epsilon_0 \right] \subset 
\Psi ([x_\beta \colon \sigma_0 < \beta <\omega_1]^\bot)$. We repeat the analogous considerations with 
$\sigma_0$ in place of $ \left(\bigvee_{\eta<\gamma} \beta_\eta \right) +\omega $ to obtain 
$\sigma_1$ with similar properties. Proceeding in this manner results in an increasing sequence 
$\sigma_n$, $n<\omega$. Now, putting $\beta^{(1)}=\bigvee_{n<\omega} \sigma_n$ satisfies $\epsilon_0 < \beta^{1} < \omega_1$ and there is a natural bicontractive projection 
$[x_\beta \colon \beta < \beta^{(1)}] \oplus [x_\beta \colon \beta^{(1)} \leq \beta <\omega_1 ] \to 
[x_\beta \colon \beta < \beta^{(1)}]$. By using the basic properties of WLD spaces 
we can find $\alpha^{(1)}, \beta^{(1)} < \omega_1$ such that $\bigvee_{\eta<\gamma} \alpha_\eta < \alpha^{(1)}$ and $[x_\beta \colon \beta < \beta^{(1)}]$ is contained in the image of $P_{\alpha^{(1)}}$ 
and $[x_\beta \colon  \beta^{(1)} < \beta < \omega_1 ]$ is in the kernel of $P_{\alpha^{(1)}}$.
Then we repeat the above considerations to obtain $\epsilon_1$ such that 
$\beta^{(1)} < \epsilon_1$ and there is a bicontractive projection onto $[x_\beta \colon \beta < \epsilon_1]$.
We fix $\alpha^{(2)} < \omega_1$ such that the above subspace is in the image of $P_{\alpha^{(2)}}$.
Then towards the kernel of this map we define $\beta^{(2)}$ accordingly. Then we find $\epsilon_2 <\omega_1$
with $\beta^{(2)} <  \epsilon_2$ and a bicontractive projection onto $[x_\beta \colon \beta < \epsilon_2]$.
Now, it is easy to see that 
$P_{\gamma}' = P_{\bigvee_{n<\omega}\alpha^{(n)}} \colon \X \to 
\left[x_\beta \colon \beta < \bigvee_{n<\omega} \beta^{(n)}\right]$
is a bicontractive projection onto. By the basic properties of the PRI and the construction we have
$\ker (P_{\bigvee_{n<\omega}\alpha^{(n)}} ) \subset \bigcap_{\alpha< \bigvee_{n<\omega}\alpha^{(n)}}
\ker (P_\alpha ) \subset  \left[x_\beta \colon \bigvee_{n<\omega} \beta^{(n)} \leq \beta < \omega_1 \right]$. 
By using the basic properties of M-basis we observe that the above inclusions must hold as equalities.
Thus the M-basis is subordinated to $P_{\gamma}'$. We proceed in this manner and it is clear that the resulting sequence is a PRI on the whole of $\X$ because the M-basis is subordinated to 
it. 
\end{proof}

An adaptation of the above back-and-forth argument yields the following fact.  
\begin{proposition}
Let $\X$, $\dens(\X)=\aleph_1$, be a Banach space with property $(\sigma)$. Then the optimal projection constants of projectional sequences on $\X$ are attained. For instance, let $\{P_{\alpha}^{(n)}\}_{\alpha<\omega_1}^{n<\omega}$ be a countable sequence of projectional 
sequences on $\X$ with 
\[\bigwedge_n \bigvee_\alpha \|P_{\alpha}^{(n)} \| = \bigwedge_n \bigvee_\alpha \|\I- P_{\alpha}^{(n)} \| = 1.\]
Then $\X$ already admits a bimonotone PRI.\qed
\end{proposition} 
   
It is known that one may extract a long monotone basic sequence in a fairly general case (see \cite{Tal2})
from a strongly dispersed ($\SD$) sequence, i.e. $\{y_\alpha \}_{\alpha<\kappa}\subset \S_\X$ such that 
$\bigcap_{\gamma<\kappa} [y_\beta \colon \gamma<\beta<\kappa] = \{0\}$, thus including 
sequences weakly convergent to $0$. Finding a bimonotone subsequence is, however, a very different task and it appears considerably harder, although bimonotonicity, in turn, is a much weaker condition than unconditionality.

\begin{theorem}\label{thm: main}
Let $\X$ be a WLD space satisfying condition $\mathrm{(d\dash d)}$.
Suppose that $\{y_\alpha\}_{\alpha<\kappa} \subset \S_\X$ is a $\SD$ sequence. 
Then there exists a bimonotone subsequence $\{y_{\alpha_\gamma}\}_{\gamma<\kappa}$. 
\end{theorem}

\begin{proof}[Proof of Theorem \ref{thm: main}]
The proof is reduced to Theorem \ref{thm: main_PRI} below which states that $\X$ admits a bimonotone 
PRI $\{P_\sigma \}_{\sigma<\kappa}$. Since $\X$ is WLD and $\{y_\alpha\}_{\alpha<\kappa} \subset \X$
is $\SD$, by the results in \cite{Tal2} there is a subsequence $\{y_{\alpha_\beta}\}_{\beta<\kappa}$,
which is a monotone basic sequence and $y_{\alpha_\beta} \stackrel{\w}{\longrightarrow} 0$, $\beta\to \kappa$. Since $\wst\dash\dens(P_{\sigma}^* (\X^* )) \leq |\sigma|\vee \aleph_0$ and $\kappa$ is regular,
it then follows by the weak convergence of the subsequence that for each $\sigma<\kappa$ there is $\beta_\sigma < \kappa$ such that 
$y_{\alpha_\beta} \in (\I - P_\sigma )(\X)$ for $\beta_\sigma < \beta < \kappa$.
On the other hand, by using the basic properties of PRI:s and the countable tightness of the norm topology of $\X$, there is for each $\beta<\kappa$ such $\sigma^{(\beta)}<\kappa$ such that 
$y_{\alpha_\beta} \in P_{\sigma^{(\beta)}}(\X)$. These conditions ensure that we can extract 
subsequences $\{\beta_{\gamma}\}_{\gamma<\kappa},\ \{\sigma_\gamma\}_{\gamma<\kappa} \subset \kappa$ such that 
\[\{y_{\alpha_{\beta_\gamma}}\}_{\gamma \leq \mu} \subset P_{\sigma_\mu}(\X),\quad 
\{y_{\alpha_{\beta_\gamma}}\}_{\mu < \gamma <\kappa} \subset (\I- P_{\sigma_\mu})(\X),\quad \mu<\kappa.\]
The bimonotonicity of the PRI then immediately yields that $\{y_{\alpha_{\beta_\gamma}}\}_{\gamma <\kappa}$ is a bimonotone basic sequence.
\end{proof}

\section{Final remarks}
We do not know if the spaces generated by an $\SD$ sequence and with Corson's property (C) are essentially WLD in some sense. Also, we do not know, given infinite cardinals $\nu < \mu$, $\lambda < 2^\nu$, what is the least $\kappa=\kappa(\lambda,\nu,\mu)$ such that $\ell_{\nu}^\infty (\mu)$ is $(\lambda,\kappa)$-a.f. If $\mu$ is larger than $\nu$ then $\ell_{\nu}^\infty (\mu)$ is relatively 'spread out'. Thus this space can be viewed as being preconditioned towards a free structure alternative in a hypothetical combinatorial dichotomy. A similar cardinal invariant $\kappa$ involving all reflexive Banach spaces of density $\geq \kappa$ seems reasonable. In Theorem \ref{thm: spread} and Corollary \ref{cor: only} we could have managed
with a Banach lattice ideal $\mathcal{I}=\mathcal{I}_{\lambda,\kappa} \subset \ell^\infty (\mu)$, in place of 
$\ell_{\lambda}^\infty (\mu)$, such that for any subspace $\Y \subset \mathcal{I}$, $\dens(\Y)\leq \lambda$, the projection band ideal $J \subset \ell^{\infty}(\mu)$, generated by $\Y$ and
with band projection $P_J\colon \X \to J$, satisfies 
$\dens(P_J (\mathcal{I}))<\kappa$. Corollary \ref{cor: only} suggests an interesting problem regarding a possible \emph{gap} in the cardinalities of unconditional basic sequences: Is it true that each reflexive Banach space of density $\mathfrak{nc}_\mathrm{rfl}$ already admits an unconditional basic sequence of length $\omega_1$?

The class of Banach spaces with $(\sigma)$ is closed in a sense.
\begin{proposition}
Suppose that a Banach space $\X$ is $(\sigma)$-generated, i.e. there is a 
continuous linear operator $T\colon \Y \to \X$ where $\Y$ has $(\sigma)$ and $T(\Y)\subset \X$ is dense. Then $\X$ already satisfies $(\sigma)$. The property $(\sigma)$ is also inherited by closed subspaces,  
preserved in taking quotients and in infinite direct sums of spaces in the $c_0$-sense.
\end{proposition}
\begin{proof}
Let us check the first claim. In verifying the condition $(\sigma)$ for $\X$ we first recall that the condition is equivalent to the coseparibility of $\bigcap_{n<\omega} \ker(f_n)$ where $(f_n)_{n<\omega} \subset \X^*$ is arbitrary. 
Observe that $f_n \circ T \in \Y^*$ and that $\bigcap_{n<\omega} \ker(f_n \circ T)$ is coseparable since $\Y$ satisfies $(\sigma)$. Let $(y_n)_{n<\omega} \subset \Y$ be such that $(y_n ) + \bigcap_{n} \ker(f_n \circ T)\subset \Y$ is dense. 
Note that $T((y_n)) + \bigcap_n \ker(T^* f_n ))$ is dense in $\X$. Moreover, since $T(\bigcap_n \ker(T^* f_n)) \subset \bigcap_n  \ker(f_n )$, we note that $T((y_n)) + \bigcap_n \ker (f_n ) \subset \X$ is dense. Thus 
$T((y_n))/ \bigcap_n \ker (f_n )$ is dense in $\X/ \bigcap_n \ker(f_n )$.
Therefore this quotient is separable and the claim follows. The rest of the statement is easy to see.
\end{proof}

In the same vein one can verify that if a CSP space $\Y$ embeds into a Banach space
$\X$ with a strong M-basis, then $\Y$ is already separable. This settles Problem 3.3 in \cite{Tal}.

\begin{proposition}\label{prop: bi_embed}
Let $\X$ satisfy $(\sigma)$ and embed isometrically into space $\Z$ which has a reverse monotone basis of length $\omega_1$ or a $1$-suppression unconditional basis of any length. Then for each separable subspace $\Y\subset \X$ the annihilator $\Y^\bot$ $1$-norms a coseparable subspace of $\X$.
Moreover, if $\X$ is additionally WLD, then for each $\Y \subset \X$, $\dens(\Y)< \dens(\X)$, the 
annihilator $\Y^\bot$ $1$-norms a subspace $\Z \subset \X$ with $\codens(\Z)= \dens(\Y)$.
\end{proposition}
\begin{proof}[Sketch of Proof]
The argument follows similar considerations as above and in \cite{Tal2}. The crux is using the countable coefficient functionals of the given basis of $\Z$, supporting a separable subspace $\Y\subset \X \subset \Z$. Then according to $(\sigma)$ the space $\X$ contains a coseparable subspace which is annihilated by these coefficient functionals.
\end{proof}

\begin{theorem}\label{thm: equiv}
Let $\X$ be a WLD space, $\dens(\X)=\kappa$. The following are equivalent:
\begin{enumerate}
\item{For each coseparable (resp. separable) subspace $\Y\subset \X$ it holds that $\Y^{\bot\bot}\subset \X^{**}$ is coseparable (resp. separable) as well;}
\item{There is a shrinking  M-basis $\{(x_\alpha ,f_\alpha )\}_{\alpha<\kappa} $ on $\X$ such that $\overline{[x_\alpha \colon \alpha\in \Lambda ]}^{\wst}\subset \X^{**}$ is norm-separable for any 
countable subset $\Lambda \subset \kappa$.}
\item{Both $\X$ and $\X^*$ are Asplund.}
\end{enumerate}
\end{theorem}

Note that the above equivalent conditions can be viewed as a kind of strong Asplund condition. If $\X$ is coseparable in its bidual then it satisfies the above equivalent conditions.
Other easy examples include $\ell^p (\kappa,\X)$, $1<p<\infty$, with $\X$ coseparable in its bidual.
\begin{proof}
Assume that condition (1) holds. Note that if $\Y \subset \X$ is a separable subspace, then according to the assumption we have that $\Y^{\bot\bot}\subset \X^{**}$ is separable. This is isometrically the dual of $\Y^*$ which in turn must be separable. Thus $\X$ is an Asplund space.

Consequently, $\X$ admits a shrinking M-basis 
$\{(x_\alpha ,f_\alpha )\}_{\alpha<\kappa}$, since this condition is equivalent to being a WLD Asplund space, see \cite[Thm. 7.12.]{Z}.
 
Suppose that $\Z \subset \X^*$ is a separable subspace. Then by the shrinking property of the M-basis there
is a countable subset $\Gamma \subset \kappa$ such that 
$\Z\subset [f_\gamma \colon \gamma\in \Gamma]$. Clearly $\{x_\gamma \}_{\gamma\in \Gamma}$
separates $[f_\gamma \colon \gamma\in \Gamma]$ so that 
$\Z^* \subset \overline{ [x_\gamma \colon \gamma\in \Gamma]}^{\wst} \subset \X^{**}$ is norm-separable. 
Thus $\X^*$ is Asplund. Hence (1) implies (3). 

Now it is easy to see the equivalence of conditions (2) and (3). The fact that (2) implies (1) is seen as follows:
Suppose that $\Y \subset \X$ is a coseparable subspace. Then there is a countable set $\Gamma\subset \kappa$ such that $[x_\gamma \colon \gamma\in \kappa\setminus \Gamma] \subset \Y $. 
Clearly $\X^{**} = \overline{\overline{ [x_\gamma \colon \gamma\in \Gamma]}^{\wst} + \Y^{\bot\bot}}$
where 
\[\left(\overline{ [f_\gamma \colon \gamma\in \Gamma]}^{\wst} \right)^* =\overline{ [x_\gamma \colon \gamma\in \Gamma]}^{\wst} \subset \X^{**} \] 
is norm-separable.
This means that the bidual $\Y^{\bot\bot}\subset \X^{**}$ is coseparable. The preservation of separability in passing to 
the bidual is seen in the same way, so that we have that (2) implies (1). 
\end{proof}

Suppose that $\X$ satisfies the equivalent conditions of Theorem \ref{thm: equiv}. 
Although it is easy to see that $\X^*$ is a Plichko space, we do not know whether it must be WLD or if the $\wst$-separable and norm-separable subspaces of $\X^{**}$ coincide.



\subsection*{Acknowledgements}

I am indebted to Professors A. Avil\'es, W. Kubi\'s and A. Plichko for their useful comments, while any mistakes are my own.

This work has been financially supported by research grants from the V\"ais\"al\"a foundation, the Finnish Cultural Foundation and the Academy of Finland Project \# 268009.

\end{document}